\documentclass[oneside]{amsart}
\usepackage[letterpaper, total={6.5in, 8.5in}]{geometry}
\newtheorem{theorem}{Theorem}[section]
\newtheorem{lemma}[theorem]{Lemma}
\newtheorem{proposition}[theorem]{Proposition}
\newtheorem{corollary}[theorem]{Corollary}

\theoremstyle{definition}

\newtheorem*{theorem*}{Theorem}

\newtheorem*{T1}{\textbf{Theorem~\ref{main theorem euler}}}
\newtheorem*{T2}{\textbf{Theorem~\ref{main theorem axisymmetric}}}

\theoremstyle{remark}
\newtheorem{remark}[theorem]{Remark}

\numberwithin{equation}{section}

\newcommand{\abs}[1]{\lvert#1\rvert}

\usepackage{amsmath}
\usepackage{amsfonts}
\usepackage{amssymb}
\usepackage[dvipsnames]{xcolor}
\usepackage{mathrsfs}
\usepackage{enumerate}

\newcommand{\norm}[1]{\left\lVert#1\right\rVert}
\newcommand{\ip}[1]{\langle #1 \rangle}

\newcommand{\R}{\mathbb{R}}
\newcommand{\T}{\mathbb{T}}

\newcommand{\Z}{\mathbb{Z}}

\newcommand{\D}{\mathscr{D}}
\newcommand{\G}{\mathscr{G}}
\newcommand{\n}{\noindent}
\newcommand{\g}{\mathfrak{g}}
\newcommand{\rot}{\text{rot}}
\newcommand{\symp}{\nabla^\perp}

\newcommand{\A}{\mathscr{A}}
\DeclareMathOperator{\curl}{curl}
\DeclareMathOperator{\Ad}{Ad}
\DeclareMathOperator{\ad}{ad}

\begin{document}

\title[Two-point boundary value problems]{Two-point boundary value problems on diffeomorphism groups}

\author{Patrick Heslin}
\curraddr{Dept. of Mathematics, Florida State University, Tallahassee, FL 32304}
\email{pheslin@fsu.edu}


\date{}


\begin{abstract}
We consider a variety of geodesic equations on Sobolev diffeomorphism groups, including the equations of ideal hydrodynamics. We prove that solutions of the corresponding two-point boundary value problems are precisely as smooth as their boundary conditions. We further utilise this regularity property to construct continuously differentiable exponential maps in the Frech{\'e}t setting.
\end{abstract}

\maketitle

\section{Introduction}
The modern formulation of hydrodynamics goes all the way back to Euler in 1757 from whom we acquire the following equations describing the motion of an ideal (incompressible \& inviscid) fluid in a closed $n$-dimensional Riemannian manifold $M$:
\begin{equation} \label{Euler}
\begin{cases}
	\partial_t u + \nabla_uu = -\nabla p \\
	\text{div}(u) = 0\\
\end{cases}
\end{equation}
If we further impose the condition that $u(0)=u_0$ we obtain the Cauchy problem for \eqref{Euler} whose rigorous studies date back to the 1920s with G{\"u}nther and Lichtenstein, the 1930s with Wolibner and the 1960s with Yudovich and Kato. The list goes on. Many of these references can be found in the monographs of Majda and Bertozzi \cite{majda}, Arnold and Khesin \cite{ak} and Bahouri, Chemin and Danchin \cite{bcd}.

\par In Arnold's seminal paper \cite{a} he observed that solutions of \eqref{Euler} can be viewed as geodesics of a right-invariant $L^2$ metric on the group $\D_\mu(M)$ of volume-preserving diffeomorphisms (sometimes referred to as \textit{volumorphisms}). In essence, this approach showcases the natural framework in which to tackle \eqref{Euler} from the Lagrangian viewpoint. In their celebrated paper Ebin and Marsden \cite{em} provided the formulation of the above in the $H^s$ Sobolev setting. Among other things they proved that for $s > \frac{n}{2}+1$ the space of volumorphisms can be given the structure of a smooth, infinite dimensional Hilbert manifold. They showed that when equipped with Arnold's $L^2$ metric the geodesic equation on this manifold is a smooth ordinary differential equation. They then applied the classic iteration method of Picard to obtain existence, uniqueness and smooth dependence on initial conditions. In particular, the last property allows one to define a smooth exponential map on $\D^s_\mu(M)$ in analogy with the classical construction in finite dimensional geometry. Hence, the work of Arnold, Ebin and Marsden enables us to explore the Riemannian geometry of fluid motion.
\par One of the central results in \cite{em} says: Let $\gamma(t) \in \D^{s}_\mu(M)$ be an $L^2$ geodesic. If $\gamma(0) \in \D^{s+1}_\mu(M)$ and $\dot{\gamma}(0) \in T_e\D^{s+1}_\mu(M)$, then $\gamma(t) \in \D^{s+1}_\mu(M)$ for all $t$ for which it was defined in $\D^s_\mu(M)$ cf. \cite[Theorem~12.1]{em} i.e., roughly speaking, if one considers \eqref{Euler} in Lagrangian coordinates as a second order ODE  in $\D^s_\mu(M)$, any solution of the corresponding Cauchy problem (an $L^2$ geodesic) is as smooth as its initial conditions. A natural question is if a similar regularity result holds when the geodesic equation is framed as a two-point boundary value problem. More precisely: \textit{Given an $L^2$ geodesic $\gamma(t)$ emanating from the identity in $\D^{s}_\mu(M)$ and, at some later time $t_0>0$, passing through $\eta \in \D^{s+k}_\mu(M)$, can it be shown that $\gamma(t)$ evolves entirely in $\D^{s+k}_\mu(M)$, i.e., is an $L^2$ geodesic as smooth as its boundary conditions?} We provide an affirmative answer to this question in the setting of 2D ideal fluids, 3D axisymmetric fluids with zero swirl, 1D equations such as the $\mu$CH and Hunter-Saxton equations, as well as the Euler-$\alpha$ equations, the symplectic Euler equations and other higher order Euler-Arnold equations arising from $H^r$ metrics with $r>0$.

\par The earliest results of this nature, to the best of the author's knowledge, are due to Constantin and Kolev in \cite{ck2002} and \cite{ck2003}. Here they used the above property to show the existence of a continuously Frech{\'e}t differentiable exponential map for right-invariant $H^{r}$ metrics, $r \in \Z_{\geq 1}$, on the group of $C^\infty$ diffeomorphisms of the circle. This result was improved upon by Kappeler, Loubet and Topalov \cite{klt1d} who showed that these exponential maps are in fact Frech{\'e}t bianalytic near zero.

\par This regularity property was further explored and used as a key ingredient by the latter authors in their study \cite{klt} of the group of orientation-preserving diffeomorphisms of the 2-torus, equipped with various right-invariant $H^r$ metrics, again for $r \in \Z_{\geq 1}$. As in \cite{ck2002} and \cite{ck2003}, their goal was to show the existence of a continuously Frech{\'e}t differentiable exponential map on the group of $C^{\infty}$ diffeomorphisms of the 2-torus.
\par {More recent work includes that of Bruveris \cite{bruveris}, who obtained a similar result to the above on the full diffeomorphism group of an arbitrary closed manifold. He shows, for a smooth right-invariant metric on $\D^s(M)$ with smooth exponential map, a $H^r$ geodesic is as smooth as its boundary points, provided that they are not conjugate.}

\par {As far as the author is aware, our results are the first of their nature pertaining to the Euler equations of ideal hydrodynamics. We record some examples of these here:}
\begin{theorem}\label{main theorem euler}
Let $s > 6$ and consider the group of volume-preserving diffeomorphisms of $\T^2$ equipped with a right-invariant $L^2$ metric. Given $u_0 \in T_e\D^s_\mu$ let $\gamma(t)$ denote the corresponding $L^2$ geodesic. If at time $t=1$, $\gamma$ passes through a point $\eta \in \D^{s+1}_\mu$, then we have $u_0 \in T_e\D^{s+1}_\mu$ and consequently $\gamma(t)$ evolves entirely in $\D^{s+1}_\mu$.
\end{theorem}

\begin{theorem}\label{main theorem axisymmetric}
Let $s>\frac{13}{2}$ and consider the group of axisymmetric diffeomorphisms of $\T^3$ with respect to any of the Killing fields $K=\partial_i$, equipped with a right invariant $L^2$ metric. Given $u_0 \in T_e\A^s_{\mu,0}$ let $\gamma(t)$ denote the corresponding $L^2$ geodesic. If at time $t=1$, $\gamma$ passes through a point $\eta \in \A^{s+1}_{\mu,0}$, then we have $u_0 \in T_e\A^{s+1}_{\mu, 0}$ and consequently $\gamma(t)$ evolves entirely in $\A^{s+1}_{\mu,0}$.
\end{theorem}
\par An immediate consequence of these theorems is that isochronal flows in these settings must either be smooth or of `low-regularity' in the Sobolev sense. We further use these regularity properties to construct continuously Frech{\'e}t differentiable exponential maps in the smooth setting. 
\begin{theorem}\label{frechet T2}
The Frech{\'e}t manifold $\D_\mu(\T^2)$ of smooth diffeomorphisms of $\T^2$ equipped with the $L^2$ metric admits a well-defined exponential map which is a local $C^1_F$-diffeomorphism at the identity.
\end{theorem}
\begin{theorem}\label{frechet T3}
The Frech{\'e}t manifold $\A_{\mu,0}(\T^3)$ of smooth axisymmetric swirl-free diffeomorphisms of $\T^3$ with respect to any of the Killing Fields $K=\partial_i$ equipped with the $L^2$ metric admits a well-defined exponential map which is a local $C^1_F$-diffeomorphism at the identity.
\end{theorem}
Later sections will provide more precise statements of these results and notations.
\par Theorem \ref{main theorem euler} improves upon a similar result of Omori \cite{omori1973} which is infinitesimal in character. His method involves constructing a connection on the full diffeomorphism group which turns the volume-preserving diffeomrphisms into a totally geodesic submanifold. In the proof of Theorem \ref{main theorem axisymmetric} we make use of the recent work of Lichtenfelz, Misio{\l}ek and Preston \cite{lmp} on the Euler equations in 3 dimensions with axisymmetric swirl-free initial data.

\par Throughout the arguments we will consider multiple configuration spaces: the full diffeomorphism group, the space of volumorphisms and its subspace of axisymmetric diffeomorphisms and the space of symplectomorphisms, all equipped with various right-invariant metrics. In each setting we will consider Sobolev diffeomorphisms of class $H^s$ with varying requirements on the value of $s$ and we may write $\G^s(M)$ to refer to any of the aforementioned spaces. Similarly, we use $\G(M)$ to refer to the corresponding smooth settings.

\par The paper is organised as follows: Sections \ref{Background1} \& \ref{Background2} contain the necessary background information for the spaces we will be considering. Section \ref{main} contains the main results. {We begin in Section \ref{diffeomorphisms} with the setting of compressible fluids on $\T^n$. The theorems presented here generalise to arbitrary dimensions the results from one and two dimensions contained in \cite{ck2002}, \cite{ck2003} and \cite{klt} without any restrictions on whether or not the endpoints are conjugate, unlike the result contained in \cite{bruveris}.} We then proceed to the Euler equations on $\T^2$ in Section \ref{volumorphisms}, where we also cover the cases of the Euler-$\alpha$ and higher order Euler-Arnold equations mentioned above. This is followed by the axisymmetric results for 3D fluids in Section \ref{axisymmetric}. The final case of the Symplectic Euler equations on $\T^{2k}$ is covered in Section \ref{symplectomorphisms}. In Section \ref{smooth} we use our results from the Section \ref{main} to construct continuously differentiable exponential maps in the Frech{\'e}t settings.

\section{Manifold Structure of Sobolev Diffeomorphisms}\label{Background1}
\par Here we gather some basic facts about diffeomorphism groups. Further details concerning these spaces can be found in Ebin and Marsden \cite{em}, Arnold and Khesin \cite{ak}, Ebin \cite{ebin2012geodesics}, Omori \cite{omori}, Misio{\l}ek and Preston \cite{mp} and Inci, Kappeler and Topalov \cite{kit}.
\subsection{Sobolev Diffeomorphisms.} Let $M$ be a compact, $n$-dimensional Riemannian manifold without boundary, with metric $g$ and volume form $\mu$. Let $g^\flat$ and $g^\sharp$ denote the usual ``musical isomorphisms".
We define the Hodge Laplacian on $k$-forms by $\Delta = d\delta + \delta d$, where $d$ is the exterior derivative, $\delta=(-1)^{n(k+1)+1}\star d \star $ is its $L^2$ dual, and $\star$ is the Hodge Star operator.
\par We define $H^s(TM)$ to be the space of vector fields on $M$ with $L^2$ derivatives up to order $s$. We equip $H^s(TM)$ with a $H^s$ inner product via:
$$\ip{u,v}_s = \sum_{k \leq s} \ip{u,(1 + \Delta)^s v}_{L^2}$$
\par In this paper we will concern ourselves with the case $s \in \Z_{\geq 0}$, but many of the constructions can be extended to the fractional case.
\par Recall that if $s > \frac{n}{2}$, then $H^s(M,M)$ is a Hilbert manifold modeled on $H^s(TM)$ equipped with $\ip{ \ , \ }_s$. If we further require $s > \frac{n}{2} + 1$, then $\D^s(M)$ inherits a smooth submanifold structure as an open subset of $H^s(M,M)$. Its tangent space at the identity is $T_e\D^s(M) = H^s(TM)$ whose dual enjoys an $L^2$-orthogonal decomposition by the Hodge theorem:
\begin{equation}\label{hodge}
    H^s(T^*M) = \mathcal{H} \oplus d \delta H^{s+2}(T^*M) \oplus \delta d H^{s+2}(T^*M)
\end{equation}
\n where $\mathcal{H}$ denotes the finite dimensional subspace of harmonic one-forms on $M$.
\par In what follows it will be convenient to define the projections:
\begin{equation}\label{mean-zero and harmonic projections}
    \pi_0 : T_e\D^s \rightarrow d \delta H^{s+2} \oplus \delta d H^{s+2} \quad \text{and} \quad \pi_\mathcal{H}: T_e\D^s \rightarrow \mathcal{H}
\end{equation}
\subsection{Volume-Preserving Diffeomorphisms} The volumorphism group is defined in terms of the Riemannian volume form $\D^s_{\mu}(M) = \{ \eta \in \D^s(M) \ \vert \ \eta^*\mu = \mu \}$. It is a smooth Hilbert submanifold of $\D^s(M)$ and the tangent space at the identity consists of all divergence-free $H^s$ vector fields $T_e\D^s_{\mu}(M) = \{ u \in T_e\D^s(M) \ \vert \ \text{div} \ u = 0 \} = g^\sharp\{\mathcal{H} \oplus \delta d H^{s+2}(T^*M)\}$, i.e the first and third summand of the Hodge Decomposition above.
\subsection{Axisymmetric Diffeomorphisms} Let $M$ be a 3-dimensional manifold equipped with a smooth Killing field $K$. Following \cite{lmp} we define a divergence-free vector field $u$ on $M$ to be \textit{axisymmetric} if it commutes with the Killing field: $[K, u]=0$. We denote the set of $H^s$ axisymmetric vector fields by $T_e\mathscr{A}_\mu^{s}(M)$. 
\par A volume-preserving diffeomorphism of $M$ is said to be \textit{axisymmetric} if it commutes with the flow of the Killing field $K$. The set of all such $H^s$ volumorphisms, $\mathscr{A}_\mu^s(M)$ is a topological group as well as a smooth totally geodesic Hilbert submanifold of $\D_\mu^s(M)$, cf. \cite[Section 3]{lmp}.
\par Axisymmetric fluid flows are of great interest and their behaviour might be informally described as $2\frac{1}{2}$-dimensional fluids, cf. \cite{uy}.
\subsection{Symplectomorphisms} Let $M$ be a symplectic manifold of dimension $2k$ and let $\omega^\flat$ and $\omega^\sharp$ denote the standard ``$\omega$-musical isomorphisms". Analogously to volumorphisms, the symplectomorphism group $\D^s_\omega(M)$ is a closed Hilbert submanifold of $\D^s(M)$ consisting of those diffeomorphisms which preserve the symplectic form $\omega$ under pullback. The tangent space at the identity is $T_e\D^s_{\omega}(M) = \{ u \in T_e\D^s(M) \ \vert \  \ d\omega^\flat u = 0 \} = \omega^\sharp \big( \mathcal{H} \oplus d \delta H^{s+2}(T^*M) \big)$, cf. \eqref{hodge}.
\par For our purposes we will further require that $g$ and $\omega$ are compatible, that is, the map $J : = g^\sharp \omega^\flat : TM \rightarrow TM$ satisfies $J^2 = -Id$. In this case $J$ is said to give $M$ an almost complex structure, cf. \cite{em}, \cite{ebin2012geodesics}, \cite{benn}, \cite{polterovich}. 
\section{Lie Group Structure and Geodesics on Diffeomorphism Groups}\label{Background2}
\par {Here we develop an infinite dimensional Lie group framework}\footnote{{It is important to note that, strictly speaking, these Hilbert manifolds of finite regularity mappings are not Lie Groups. However, they are topological groups and possess sufficient geometric structure for our purposes. See Remark \ref{not Lie Groups} for more details.}} for $\D^s(M)$ and its submanifolds of interest. We use this to present geodesic equations on diffeomorphism groups in a general formulation for a variety of metrics. The construction follows closely that of \cite{em} and \cite{mp}.
\par Throughout the various arguments in Section \ref{main} we will deal with products of $H^s$ Sobolev functions and compositions with elements of $\D^s(M)$, so it is crucial that we have some control over the regularity of these objects. {For this reason we recall the following results cf. \cite{adams}, \cite{em}, \cite{kit}.}
\begin{lemma}\label{sobolev}
For any $s > \frac{n}{2} + 1$, $\abs{m} \leq s$ and $k \geq 0$ we have:
\begin{enumerate}
    \item $H^{s}(M,\R) \times H^{m}(M, \R) \rightarrow H^{m}(M, \R) \ ; \  (u,v) \mapsto uv$ is bounded.\\
    \item $H^{s+k}(M,\R^d) \times \D^{s}(M) \rightarrow H^{s}(M, \R^d) \ ; \ (v,\phi) \mapsto v\circ \phi$ is $C^k$-smooth.\\
    \item $\D^{s+k}(M) \rightarrow \D^{s}(M) \ ; \ \phi \mapsto \phi^{-1}$  is $C^k$-smooth.
\end{enumerate}
\end{lemma}
In particular this gives us that $\D^s(M)$ is a topological group (one can readily show that $\D^s_\mu(M)$, $\A^s_\mu(M)$ and $\D^s_\omega(M)$ are all subgroups), where right translation $R_\eta$ is smooth and left translation $L_\eta$ is continuous (although not even Lipschitz continuous) in the $H^s$ topology.
\par Again, using $\G^s(M)$ to refer to any of the aforementioned manifolds, we denote the (almost) Lie algebra by $\g^s = T_e\G^s(M)$, we recall the group adjoint:
\begin{align*}
    \Ad_\eta v &= d_{\eta^{-1}}L_\eta d_eR_{\eta^{-1}} v \\
    &= \big(D\eta \cdot v \big) \circ \eta^{-1}
\end{align*}
and the Lie algebra adjoint:
$$\ad_u v = -[u,v]$$
\n where $\eta(t)$ is any curve in $\G^s(M)$ with $\eta(0) = e$ and $\dot{\eta}(0) = u$. Observe that if $u,v \in T_e\G^{s}$, their commutator is a priori only of Sobolev class $H^{s-1}$.
\begin{remark}\label{not Lie Groups}
{The derivative loss coming from left translation and the Lie bracket are examples of why these diffeomorphism groups are not Lie Groups. However, the group structure they possess is sufficient for our purposes.}
\end{remark}

\par We now equip $\G^s(M)$ with a (weak\footnote{We say a metric on $\G^s$ is \textit{weak} if it induces a weaker topology than the inherent $H^s$ topology.}) right-invariant, Riemannian metric $\ip{\cdot , \cdot}$. Note that the definitions of $\Ad_\eta$ and $\ad_u$ depend only on the group structure, and not on the choice of the metric. Hence, the geometry is, in some sense, encoded in the coadjoint operators defined as follows:
$$\ip{\Ad^*_\eta u , v} = \ip{u, \Ad_\eta v} \quad \text{and} \quad \ip{\ad_u^*v, w} = \ip{v, \ad_u w }$$
for all $u, v, w \in \g^s$.
\n Throughout this paper $*$ will always refer to the adjoint of an operator with respect to the relevant metric and configuration space; which should be clear from the context.
\par We are specifically interested in the case where we equip $\G^s(M)$ with a right-invariant $H^r$ metric with $r\geq 0$ defined at the identity by:
\begin{equation}\label{inertia}
\ip{u, v}_{H^r} := \ip{A^ru,v}_{L^2} \qquad \quad u, v \in T_e\G^s
\end{equation}
\n where $A^r$ denotes an invertible elliptic pseudo-differential operator of order $2r$; e.g. $A^r = (1 + g^\sharp \Delta g^\flat)^r$, cf. Taylor \cite{taylorpseudo}. We refer to such an $A^r$ as an \textit{inertia operator} {and denote its inverse by $A^{-r}$}. We will always assume that $A^r$ commutes with both $d$ and $\delta$ and that we have at least $s > \frac{n}{2} + 2r + 1$ to guarantee a baseline level of control over the regularity of the vector fields involved in the later calculations. We will use $*_r$ to denote the adjoint of an operator with respect to such a $H^r$ metric.
\par We define a geodesic on $\big( \G^s(M), \ip{\cdot, \cdot} \big)$ to be a critical path for the energy functional induced by the metric $\ip{\cdot, \cdot}$ and recall two important lemmas pertaining to geodesics on Lie groups with right-invariant metrics:

\begin{lemma}[cf. {\cite[Theorem~3.2]{mp}}]\label{Lie group geodesic}
If $\G$ is a Lie group equipped with a (possibly weak) right-invariant metric $\ip{\cdot,\cdot}$, then a curve $\sigma(t)$ is a geodesic if and only if the curve $u(t)$ in the Lie algebra given by the flow equation:
\begin{equation*}
    \dot{\sigma}(t) = d_{e}R_{\sigma(t)} u(t)
\end{equation*}
solves the Euler-Arnold equation:
\begin{equation}\label{Lie group geodesic equation}
    \partial_t u (t) = - \ad^*_{u(t)}u(t)
\end{equation}
\end{lemma}

\begin{lemma}[cf. {\cite[Corollary~3.3]{mp}}]\label{Lie group conservation law}
If $\sigma(t)$ is a curve in $\G$ with $u(t) = d_{\sigma(t)}R_{\sigma(t)^{-1}}(\dot{\sigma}(t))$ satisfying \eqref{Lie group geodesic equation} with initial conditions $\sigma(0) = e$ and $u(0)=u_0$, then we have the following conservation law:
\begin{equation}\label{Lie group conservation law equation}
    \Ad^{*}_{\sigma(t)}u(t) = u_0
\end{equation}
and hence we can rewrite the flow equation as:
\begin{equation*}
    \dot{\sigma}(t) = d_eR_{\sigma(t)} \Ad^{*}_{\sigma(t)^{-1}}u_0 =d_eL^{*}_{\sigma(t)^{-1}}u_0 
\end{equation*}
\end{lemma}
\begin{remark}
It is important to note at this point that Lemmas \ref{Lie group geodesic} \& \ref{Lie group conservation law} do not apply seamlessly to our setting of $H^s$ Sobolev diffeomorphism groups. As mentioned earlier, $\D^{s}(M)$ and its subgroups of interest are not Lie groups, on account of left translation only being continuous, etc. However, in any of the settings we consider in this paper, analogues of \eqref{Lie group geodesic equation} and \eqref{Lie group conservation law equation} will hold, {cf. Section 3 of \cite{mp}}.
\end{remark}

\par Notable examples of Euler-Arnold equations \eqref{Lie group geodesic equation} in the setting of diffeomorphism groups include the incompressible Euler equations in two and three dimensions, Burgers' equation, the Hunter-Saxton equation, the Camassa-Holm equation, the $\mu$CH equation as well as the Euler-$\alpha$ equations and the symplectic Euler equations. The associated Riemannian geometry of these equations as well as the existence of a $C^{\infty}$ exponential map and its properties has been studied extensively in the literature cf. \cite{em}, \cite{omori}, \cite{shnirelman1994}, \cite{emp}, \cite{shkoller}, \cite{mp}, \cite{ck2002}, \cite{ck2003}, \cite{klt}, \cite{ebin2012geodesics}, \cite{khesinsymplectic} and many others.

\section{Main Results}\label{main}
Let $\G^{s}$ be a group of diffeomorphisms of $M$ equipped with a (weak) Riemannian metric $\ip{\cdot , \cdot}$ and an associated exponential map where we assume that the results stated in Lemmas \ref{Lie group geodesic} \& \ref{Lie group conservation law} hold. Our main focus is the following question: if a geodesic $\gamma(t)$ emanating from the identity in $\G^{s}$ at some later time $t_0>0$ passes through $\eta \in \G^{s+k}$, can it be shown that $\gamma(t)$ evolves entirely in $\G^{s+k}$? For our purposes, it will be sufficient to assume that both $t_0=1$ and $k=1$. {References for the various commutator estimates involved in the calculations include Kato-Ponce \cite{kp} and Taylor \cite{taylorpseudo}.}

\par As mentioned in the introduction, previous results in this vein are due to Constantin and Kolev \cite{ck2002}, \cite{ck2003} and Kappeler, Loubet and Topalov \cite{klt}, who worked with compressible equations on the circle and the 2-torus respectively and Bruveris \cite{bruveris}. At the heart of our method lie the various conservation laws motivated by Lemma \ref{Lie group conservation law}.

\par In this paper we will concern ourselves primarily with the flat case $M = \T^n$. However, many of the constructions can be extended to the setting of curved spaces by collecting any lower order terms arising in the various calculations due to derivatives of the components of the metric and its Christoffel symbols on $M$ into a single term which will be negligible for our purposes. {It should be noted that we do not believe our results are sharp. We require varying conditions on the parameter $s$ due to the method of proof; all of which can be explained by Lemma \ref{sobolev}.}
\par We begin the groundwork for the main results. We aim to establish an explicit relationship between the regularity of the end configuration $\eta$ and the regularity of the intial velocity $u_0$. To do this we will make use of \eqref{Lie group conservation law equation}, suitably modified to each case considered below.

\begin{lemma}\label{workhorse}
Let $s>\frac{n}{2}+3$ and let $\gamma(t)$ be a smooth curve in $\G^s(\T^n)$ with $v(t) = d_{\gamma(t)} R_{\gamma(t)^{-1}}\big( \dot{\gamma}(t) \big) $, $\gamma(0) = e$ and $\gamma(1)=\eta$. We have the following identity
\begin{equation}\label{relationship}
    \Delta \eta = D\eta \int_{0}^{1} D\gamma(t)^{-1}  (P_\lambda v)\circ \gamma (t) \ dt + G
\end{equation}
where $G = G(\gamma)$ is of class $H^{s-1}$, {$\Delta \eta$ is defined by considering $\eta : \T^n \rightarrow \T^n$ in coordinates and applying $\Delta$ to each component and, for any $\lambda \in \R$, $P_\lambda$ is a differential operator acting component-wise given by:}
\begin{equation}
	P_\lambda(t) = \lambda - \sum_{i,j=1}^{n}p_{ij}(t)\partial_i\partial_j \quad \text{with} \quad p_{ij}(t) = \big( D\gamma  D\gamma^{\top} \big)_{ij} \circ \gamma^{-1}(t) 
\end{equation}
with $D\gamma^{\top}$ denoting the pointwise adjoint of $D\gamma$.
\end{lemma}
\begin{proof}Applying the Laplacian $\Delta$ to the tangent vector to the curve $\gamma(t)$, we get:
\begin{align*}
    \Delta \big(\frac{d \gamma}{dt}\big) &= \Delta(v\circ \gamma) = (Dv\circ \gamma )  \Delta \gamma - \lambda v \circ \gamma + \big(P_\lambda v\big) \circ \gamma
\end{align*}
Rearranging, we get:
\begin{equation}\label{laplacian}
    \begin{cases}
    \dfrac{d}{dt}\big(\Delta \gamma\big) - (Dv \circ \gamma)(\Delta \gamma) = (P_\lambda v ) \circ \gamma - \lambda v \circ \gamma\\
    \Delta \gamma (0) = 0. \\
    \end{cases}
\end{equation}
\n On the other hand, differentiating the flow equation $d \gamma/d t = v \circ \gamma$ in the spatial variables gives:
\begin{equation}\label{transport}
	\begin{cases}
	\dfrac{d}{dt}(D\gamma) - Dv \circ \gamma  D\gamma = 0 \\
	D\gamma(0)= \mathrm{Id}.\\
	\end{cases}
\end{equation}

\n Using \eqref{transport} and the Duhamel formula, we can now rewrite \eqref{laplacian} as an integral equation in the form:
\[
\Delta \gamma(t) = D\gamma(t) \int_{0}^{t} D\gamma(\tau)^{-1}  (P_\lambda v )\circ \gamma (\tau) \ d\tau + G(t)
\]
where 
$$G(t) = -\lambda D\gamma(t) \int_{0}^{t} D\gamma(\tau)^{-1}(v\circ \gamma)(\tau) \ d\tau$$
is a curve in $H^{s-1}(\T^n, \mathbb{R}^n)$. Evaluating at $t=1$ and denoting $G(1)$ by $G$ we arrive at \eqref{relationship}.
\end{proof}
\par Suppose now that $\gamma(t)$ is a geodesic of the metric $\ip{\cdot, \cdot}$ in $\G^s$ with $\gamma(0)=e$ and $\dot{\gamma}(0)=u_0 \in T_e\G^s$. Suppose further that at time $t_0=1$, $\gamma$ passes through $\eta$. Since $\Ad^{-1}_{\gamma(t)} = D\gamma(t)^{-1}  R_{\gamma(t)}$, using the conservation law \eqref{Lie group conservation law equation} and Lemma \ref{workhorse}, we may rewrite \eqref{relationship} as:
\begin{equation}\label{conjugation relationship}
    \Delta \eta = D\eta \int_{0}^{1} \Ad_{\gamma(t)}^{-1}  P_\lambda (t)  \big(\Ad^{-1}_{\gamma(t)}\big)^{*} u_0 \ dt + G
\end{equation}
where $*$ denotes the metric adjoint as before.
\par We can see from \eqref{conjugation relationship} that $P_\lambda$ will play a central role in establishing a relationship between the regularity of $\eta$ and our initial data $u_0$. To this end we establish that it defines a norm equivalent to the $H^1$ norm for sufficiently large $\lambda>0$.

\begin{lemma}\label{elliptic estimate} There exists a $\lambda > 0$ such that, for any $t \in [0,1]$ and $v \in H^{1}$, the operator $P_\lambda$ satisfies the estimate:
\begin{equation*}
\langle P_\lambda(t) v, \ v \rangle _{L^2} \  \simeq \norm{v} _{H^1}^2
\end{equation*}
{with the constants depending on the curve $\gamma : [0,1] \rightarrow \D^s_\mu(\T^n)$.}
\end{lemma}
\begin{proof}
We first derive an estimate for the coefficients $p_{ij}$. For any $w \in \R^n$, $t \in [0, 1]$ and $1 \leq i,j \leq n$, we have
\begin{align*}
    \sum_{i,j = 1}^n p_{ij}(t) w_i w_j &= w^{\top}  [D\gamma(t)][D\gamma(t)^{\top}] \circ \gamma^{-1}(t)  w \\
    &= \abs{D\gamma(t)^{\top} \circ \gamma^{-1}  w}^2.
\end{align*}
\n As $s>\frac{n}{2} + 3$, it follows from a compactness argument and the fact that $D\gamma \circ \gamma^{-1}$ is a linear isomorphism of $\R^n$ for all $t\in [0, 1]$ that:
\begin{equation}\label{pij}
    \sum_{i,j = 1}^n p_{ij}(t) w_i w_j \ \simeq \vert w \vert^2
\end{equation}
Integrating by parts and using estimate \eqref{pij}, we have:
    \begin{align*}
        \langle P_\lambda(t) v, \ v \rangle _{L^2} &= \lambda \ip{ v, v }_{L^2} - \sum_{i,j}\ip{ p_{ij}(t)\partial_i \partial_j v ,  v }_{L^2} \\
        &= \lambda \norm{v}_{L^2}^2 + \frac{1}{2} \int_{\T^n}\sum_{i,j} \partial_i p_{ij}(t) \partial_j \vert v \vert^2 \ dx + \int_{\T^n} \sum_{i,j}p_{ij}(t) \partial_jv  \partial_iv \ dx \\
        &= \lambda \norm{v}_{L^2}^2 - \frac{1}{2} \sum_{i,j}\ip{\partial_j\partial_i p_{ij}(t) v, v}_{L^2} + \int_{\T^n} \sum_{i,j}p_{ij}(t) \partial_jv  \partial_iv \ dx \\
        &\gtrsim \lambda \norm{v}_{L^2}^2 - \frac{1}{2} \sum_{i,j}\norm{\partial_j\partial_i p_{ij}(t) v}_{L^2}\norm{ v}_{L^2} + \int_{\T^n} \sum_{i} \vert \partial_iv_1 \vert^2 + \vert \partial_iv_2 \vert^2 \ dx \\
        &\gtrsim \norm{v}_{H^1}^2
    \end{align*}
	\n where the penultimate and last inequalities follow from the uniform boundedness in $t \in [0, 1]$ of the coefficients $\partial_i\partial_j p_{ij}$ and taking $\lambda$ sufficiently large.
	\par On the other hand using \eqref{pij} again we have
    \begin{align*}
        \langle P_\lambda(t) v, \ v \rangle _{L^2} &\lesssim \lambda \norm{v}_{L^2}^2 + \frac{1}{2} \sum_{i,j}\norm{\partial_j\partial_i p_{ij}(t) v}_{L^2}\norm{ v}_{L^2} + \int_{\T^n} \sum_{i} \vert \partial_iv_1 \vert^2 + \vert \partial_iv_2 \vert^2 \ dx \\
        & \lesssim \norm{v}_{H^1}^2.
    \end{align*}
\end{proof}
 
\subsection{Diffeomorphisms of the Flat $n$-Torus $\T^n$}\label{diffeomorphisms}
\par We are now ready to present our first theorem, which generalises to $n$ dimensions the corresponding results proved in \cite{ck2002}, \cite{ck2003} and \cite{klt}. Our method follows along the lines of the aforementioned. The main difference is the explicit use of the coadjoint operators and the conservation law \eqref{Lie group conservation law equation} which shortens the argument.

\begin{theorem}\label{main theorem compressible}
Let $r \geq 1$ be an integer, $s>\frac{n}{2}+2r+5$ and consider the space of orientation preserving diffeomorphisms of $\T^n$ equipped with a right-invariant $H^r$ metric \eqref{inertia}. Given $u_0 \in T_e\D^s$ let $\gamma(t)$ denote the corresponding $H^r$ geodesic. If at time $t=1$, $\gamma$ passes through a point $\eta \in \D^{s+1}$, then we have $u_0 \in T_e\D^{s+1}$ and consequently $\gamma(t)$ evolves entirely in $\D^{s+1}$.
\end{theorem}

\par Recall from \cite{mp}, Section 3, that the group coadjoint on $\D^s(\T^n)$ equipped with a $H^r$ metric has the form:
$$\Ad^{*_r} _\eta v = A^{-r}\big( \det(D\eta)  D\eta^{\top}  \big(A^r v \circ \eta \big)\big)$$
\par Hence we have the following version of \eqref{Lie group conservation law equation}:
\begin{equation}\label{coadjoint con diffeomorphisms 2}
    u(t) = \big(\Ad_{\gamma(t)}^{-1}\big)^{*_r} u_0 = A^{-r}  R_{\gamma(t)}^{-1}  \big(\det\big(D\gamma(t)\big)  D\gamma(t)^{\top}\big)^{-1}  A^r u_0
\end{equation}

\begin{remark}\label{main obstacle}
Later in our argument we will require that, for each $t \in [0,1]$, $\big(\Ad_{\gamma(t)}^{-1}\big)^{*_r} : T_e\D^{s+1} \rightarrow T_e\D^{s+1}$ be a bounded invertible linear operator, so it is at this point that we can explicitly see the necessity of $r\geq 1$. A simple derivative count shows that $A^r$ must be at least of order $2$ to prevent a loss of derivatives coming from the multiplication by the coefficients of the matrix $D\gamma(t)$ which are apriori only of class $H^{s-1}$ {(recall that only the endpoints of $\gamma(t)$ will be assumed to be of class $H^{s+1}$)}.
\end{remark}

\begin{proof}[Proof of Theorem ~\ref{main theorem compressible}]
The proof consists of three stages. We begin by combining \eqref{conjugation relationship} with some commutator estimates to acquire an expression of the form $\Delta \eta = D \eta M u_0 + \widetilde{G}$, where $\widetilde{G}$ is of class $H^{s-1}$. We then establish that, for $v \in H^s$, $Mv$ being of class $H^{s-1}$ implies that $v$ is of class $H^{s+1}$. We finish by concluding that, as $Mu_0 = D \eta^{-1}\big(\Delta \eta - \widetilde{G}\big) \in H^{s-1}$, we have that $u_0 \in H^{s+1}$. Hence, by the results for the initial value problem, we have that $\gamma(t)$ evolves entirely in $H^{s+1}$ and the proof will be complete. Recall \eqref{conjugation relationship}:
\begin{align*}
        \Delta \eta &= D\eta \int_{0}^{1} \Ad_{\gamma(t)}^{-1}  P_\lambda (t)  \big(\Ad^{-1}_{\gamma(t)}\big)^{{*_r}} u_0 \ dt + G
\end{align*}
where again $G$ is of class $H^{s-1}$. Introducing a commutator term $P_\lambda = A^{-\frac{r}{2}}  P_\lambda (t)  A^{\frac{r}{2}} + A^{-\frac{r}{2}}  \big[ A^{\frac{r}{2}}, P_\lambda (t)\big]$ we have
\begin{align*}
        \Delta \eta &= D\eta \int_{0}^{1} \Ad_{\gamma(t)}^{-1}  A^{-\frac{r}{2}}  P_\lambda (t)  A^{\frac{r}{2}}  \big(\Ad^{-1}_{\gamma(t)}\big)^{{*_r}} u_0 \ dt + D\eta \int_{0}^{1} \Ad_{\gamma(t)}^{-1}  A^{-\frac{r}{2}}  \big[ A^{\frac{r}{2}}, P_\lambda (t)\big]  \big(\Ad^{-1}_{\gamma(t)}\big)^{{*_r}} u_0 \ dt + G
\end{align*}
Notice now that $\big[ A^{\frac{r}{2}}, P_\lambda (t)\big] = \big[ p^{ij}, A^{\frac{r}{2}}\big] \partial_i \partial_j$ and, as $p^{ij}$ are of class $H^{s-1}$, we have that $\big[ A^{\frac{r}{2}}, P_\lambda (t)\big] : H^{s} \rightarrow H^{s-r-1}$ by {the following Kato-Ponce \cite{kp} type commutator estimate, cf. \cite{taylorpseudo}}:
\begin{equation}\label{taylor}
    \norm{A^{\frac{r}{2}}(fg) - fA^{\frac{r}{2}}g}_{H^{s-r-1}} \lesssim \norm{\nabla f}_{\infty}  \norm{g}_{H^{s-2}} + \norm{f}_{H^{s-1}} \norm{g}_{\infty}
\end{equation}
{Hence, the term $\displaystyle D\eta \int_{0}^{1} \Ad_{\gamma(t)}^{-1}  A^{-\frac{r}{2}}  \big[ A^{\frac{r}{2}}, P_\lambda (t)\big]  \big(\Ad^{-1}_{\gamma(t)}\big)^{{*_r}} u_0 \ dt$ belongs to $H^{s-1}$ and we rewrite:}
\begin{align*}
        \Delta \eta &= D\eta \int_{0}^{1} \Ad_{\gamma(t)}^{-1}  A^{-\frac{r}{2}}  P_\lambda (t)  A^{\frac{r}{2}}  \big(\Ad^{-1}_{\gamma(t)}\big)^{{*_r}} u_0 \ dt + \widetilde{G}
\end{align*}
where $\displaystyle \widetilde{G} := D\eta \int_{0}^{1} \Ad_{\gamma(t)}^{-1}  A^{-\frac{r}{2}}  \big[ A^{\frac{r}{2}}, P_\lambda (t)\big]  \big(\Ad^{-1}_{\gamma(t)}\big)^{{*_r}} u_0 \ dt + G$ is of class $H^{s-1}$. \\
\par Next, we claim that the operator defined by:
\begin{equation}\label{M}
Mv = \int_{0}^{1} M_t v \ dt := \int_{0}^{1} \Ad_{\gamma(t)}^{-1}  A^{-\frac{r}{2}}  P_\lambda (t)  A^{\frac{r}{2}}  \big(\Ad^{-1}_{\gamma(t)}\big)^{{*_r}} v \ dt
\end{equation}
is a linear isomorphism from $H^{s+1}$ to $H^{s-1}$. We can see that, for $r+1 \leq k \leq s+1$, $M : H^{k} \rightarrow H^{k-2}$ is a bounded linear operator, as the integrand is a composition of bounded linear operators. Consider the case $k = r+1$, and define the bilinear form:
\[
\Lambda : H^{r+1} \times H^{r+1} \rightarrow \R \quad ; \quad \Lambda(v,w) = \langle Mv , w \rangle_{H^r}
\] 
\n It follows from the boundedness of $M$ that $\Lambda$ is bounded. Furthermore, by Lemma \ref{elliptic estimate} we have:
\begin{align*}
	\Lambda(v,v) &= \int_0^{1} \Big\langle A^{-\frac{r}{2}} P_\lambda(t)  A^{\frac{r}{2}}  \big(\Ad^{-1}_{\gamma(t)}\big)^{{*_r}} v \ , \ \big(\Ad^{-1}_{\gamma(t)}\big)^{{*_r}} v \Big\rangle_{H^r} \ dt \\
	&= \int_0^{1} \Big\langle P_\lambda(t)  A^{\frac{r}{2}}  \big(\Ad^{-1}_{\gamma(t)}\big)^{{*_r}} v \ , \ A^{\frac{r}{2}} \big(\Ad^{-1}_{\gamma(t)}\big)^{{*_r}} v \Big\rangle_{L^2} \ dt \\
	&\gtrsim \int_0^{1} \norm{A^{\frac{r}{2}}  \big(\Ad^{-1}_{\gamma(t)}\big)^{{*_r}} v}_{H^1}^2 \ dt \\
	&\gtrsim \norm{ \big(\Ad^{-1}_{\gamma(t)}\big)^{{*_r}} v}_{H^{r+1}}^2 \\
	&\gtrsim \norm{v}_{H^{r+1}}^2
\end{align*}
Hence, by the Lax-Milgram theorem, we have that $M : H^{r+1} \rightarrow H^{r-1}$ is a linear isomorphism. We now proceed by induction. Assume $M : H^{k} \rightarrow H^{k-2}$ is a linear isomorphism for some $r+1 \leq k \leq s$, and let $g \in H^{k-1}$. By the induction hypothesis, there exists $f \in H^{k}$ such that $Mf = g$. We claim that $f$ is in fact of class $H^{k+1}$. For $j = 1,...,n$ consider:
\begin{align*}
    \partial_j f &= \partial_j M^{-1}g \\
    &= M^{-1} \partial_j g \ + \ [\partial_j, M^{-1}]g \\
    &= M^{-1} \partial_j g \ + \ M^{-1}[M, \partial_j]M^{-1}g \\
    &= M^{-1} \partial_j g \ + \ M^{-1}[M, \partial_j]f.
\end{align*}
By assumption, the first term on the RHS is in $H^{k}$. As for the commutator term, from \eqref{M} is suffices to show that for $j = 1,...,n$ and $r+1 \leq k \leq s$, $[M_t, \partial_j] : H^{k} \rightarrow H^{k-2}$. This will give us that $[M, \partial_j] : H^{k} \rightarrow H^{k-2}$, and hence $\partial_j f \in H^{k}$.
\par First, for an operator $B$, define $C_{\gamma}( B ) = R_{\gamma}  B  R_{\gamma}^{-1}$, and its inverse $\widetilde{C}_{\gamma}( B ) = R_{\gamma}^{-1}  B  R_{\gamma}$. Then, using \eqref{coadjoint con diffeomorphisms 2}, the fact that $[A^{r}, \partial_j]=0$ and some standard commutator algebra, we have:
\begin{align*}
    \big[M_t, \partial_j\big] &= \big[D\gamma  R_{\gamma}  A^{-\frac{r}{2}}  P_{\lambda}  A^{-\frac{r}{2}}  R_{\gamma}^{-1}  \big(\det\big(D\gamma\big)  D\gamma^{\top}\big)^{-1}  A^r , \partial_j\big] \\
    &= \big[D\gamma  C_{\gamma} \big( A^{-\frac{r}{2}}  P_{\lambda}  A^{-\frac{r}{2}} \big)  \big(\det\big(D\gamma\big)  D\gamma^{\top}\big)^{-1}, \partial_j\big]  A^r \\
    &= \big[D\gamma, \partial_j\big]  C_{\gamma} \big( A^{-\frac{r}{2}}  P_{\lambda}  A^{-\frac{r}{2}} \big)  \big(\det\big(D\gamma\big)  D\gamma^{\top}\big)^{-1}  A^r \\
    & \qquad + D\gamma  C_\gamma(A^{-\frac{r}{2}}  P_{\lambda}  A^{-\frac{r}{2}})  \big[\big(\det\big(D\gamma\big)  D\gamma^{\top}\big)^{-1}, \partial_j\big]  A^r \\
    & \qquad + D\gamma  \big[C_\gamma(A^{-\frac{r}{2}}  P_{\lambda}  A^{-\frac{r}{2}}), \partial_j\big]  \big(\det\big(D\gamma\big)  D\gamma^{\top}\big)^{-1}  A^r.
\end{align*}
{The first and second term in the sum} mapping to the correct space follows from the fact that, for $f \in H^{s-1}$, we have $\big[f,\partial_j\big] : H^{k-2} \rightarrow H^{k-2}$ and $H^{k-2r} \rightarrow H^{k-2r}$, which is clear. All that remains is to show that $\big[C_\gamma(A^{-\frac{r}{2}}  P_{\lambda}  A^{-\frac{r}{2}}), \partial_j\big]: H^{k-2r} \rightarrow H^{k-2}$, which is equivalent to showing $\big[A^{-\frac{r}{2}}  P_{\lambda}  A^{-\frac{r}{2}}, \widetilde{C}_\gamma(\partial_j)\big] = \big[A^{-\frac{r}{2}}  P_{\lambda}  A^{-\frac{r}{2}}, \partial_j\gamma^i \circ \gamma^{-1} \partial_i\big]: H^{k-2r} \rightarrow H^{k-2}$. Let $f_j^i = \partial_j\gamma^i \circ \gamma^{-1} \in H^{s-1}$. Now we have:
\begin{align*}
    \big[A^{-\frac{r}{2}}  P_{\lambda}  A^{-\frac{r}{2}}, f_j^k \partial_i\big]
    &= A^{-\frac{r}{2}}  \big[f^k_j, A^{\frac{r}{2}}\big]  \partial_i  A^{-\frac{r}{2}}  P_\lambda  A^{-\frac{r}{2}} + A^{-\frac{r}{2}}  \big[P_\lambda, f^k_j \partial_i\big]  A^{-\frac{r}{2}} + A^{-\frac{r}{2}}  P_\lambda  A^{-\frac{r}{2}}  \big[f^k_j, A^{\frac{r}{2}}\big]  \partial_i  A^{-\frac{r}{2}}
\end{align*}
and the result follows by observing that:
\begin{itemize}
    \item $\big[P_\lambda, f_j^k \partial_i\big]$ is a second order differential operator with coefficients in $H^{s-3}$
    \item $\big[f_j^k,A^{\frac{r}{2}}\big]: H^{k-3} \rightarrow H^{k-r-2}$
    \item $\big[f_j^k,A^{\frac{r}{2}}\big]: H^{k-r-1} \rightarrow H^{k-2r}$
\end{itemize}
where the latter two follow by commutator estimates analogous to \eqref{taylor}. Hence $\partial_j f \in H^{k}$, which implies that $f \in H^{k+1}$, and by the induction argument $M : H^{s+1} \rightarrow H^{s-1}$ is a linear isomorphism. \\
\par Finally, we have:
\begin{align*}
    \Delta \eta &=  D\eta  M (u_0) + \widetilde{G},
\end{align*}
which implies
\begin{align*}
    u_0 = M^{-1} \big( D\eta^{-1} \big( \Delta \eta - \widetilde{G} \big) \big) \in H^{s+1}
\end{align*}
and the proof of Theorem \ref{main theorem compressible} is complete.
\end{proof}

\subsection{Volume-Preserving Diffeomorphisms of the Flat 2-Torus $\T^2$}\label{volumorphisms} The main result presented in this section is as follows:
\begin{T1}
Let $s > 6$ and consider the group of volume-preserving diffeomorphisms of $\T^2$ equipped with a right-invariant $L^2$ metric. Given $u_0 \in T_e\D^s_\mu$ let $\gamma(t)$ denote the corresponding $L^2$ geodesic. If at time $t=1$, $\gamma$ passes through a point $\eta \in \D^{s+1}_\mu$, then we have $u_0 \in T_e\D^{s+1}_\mu$ and consequently $\gamma(t)$ evolves entirely in $\D^{s+1}_\mu$.
\end{T1}
\par For the proof we make use of the conservation of vorticity in 2D, which takes the form: 
\begin{equation}\label{conservation of vorticity}
    \rot (u)\circ \gamma = \rot (u_0),
\end{equation}
where $u$ is the Eulerian velocity of the flow $\gamma$, cf. \eqref{Lie group conservation law equation}. Applying the symplectic gradient $\symp := (-\partial_2, \partial_1)$ to both sides of \eqref{conservation of vorticity} we obtain the following conservation law.
\begin{equation} \label{euler con}
	\Delta u = \Ad_{\gamma} \left( \Delta u_0 \right) 
\end{equation}
Now, recall by the Hodge Decomposition Theorem \eqref{hodge} that we may decompose the tangent space $T_e\D^s_\mu$ into an $L^2$-orthogonal sum:
\begin{equation*}
    T_e\D^s_\mu = T_e\D^s_{\mu,ex} \oplus \mathcal{H}
\end{equation*}
where $T_e\D^s_{\mu,ex} = g^\sharp\big(d \delta H^{s+2}(T^{*}\T^2)\big)$ is known as the space of exact volume-preserving vector fields and again $\mathcal{H}$ denotes the space of harmonic vector fields. Hence, we can rewrite $u_0 = \symp f_0 + h_0$, where $f_0 \in H^{s+1}_0(\T^2, \R):= \big\{ g \in H^{s+1}(\T^2, \R) \ \vert \ \widehat{g}(0) = 0 \big\}$ and $h_0 \in \mathcal{H}$. So  we may reformulate \eqref{euler con} as:
\begin{lemma} For $u_0 = \symp f_0 + h$, $\gamma$ and $u$ as above, we have:
\begin{equation}\label{euler con exact}
    u(t) = \symp \big( \Delta^{-1} R_{\gamma(t)}^{-1} \Delta f_0  \big) + h(t)
\end{equation}
where $h(t)$ evolves in $\mathcal{H}$ and $\Delta^{-1} : H^{s-1}_0(\T^2, \R) \rightarrow H^{s+1}_0(\T^2, \R)$, defined in frequency space by:
\[
  \widehat{\Delta^{-1}f}(\xi) :=
  \begin{cases}
     0 & \text{if $\xi = (0,0)$} \\
     \dfrac{\widehat{f}(\xi)}{\vert \xi \vert^2} & \text{if $\xi \in \Z^2 \setminus {(0,0)}$}
  \end{cases}
\]
is a linear isomorphism.
\end{lemma}
\begin{proof}
Using \eqref{euler con} we have:
\begin{align*}
    \Delta u &= \Ad_\gamma (\Delta u_0) \\
    &= \Ad_\gamma \big(\Delta (\symp f_0 + h)\big) \\ 
    &= \Ad_\gamma(\symp \Delta f_0) \\
    & = \symp (R_\gamma^{-1} \Delta f_0)
\end{align*}
which, as $\Delta$ acts component-wise on vector fields, immediately yields \eqref{euler con exact}.
\end{proof}
Formula \eqref{euler con exact} will be more convenient for our purposes than \eqref{conservation of vorticity}.
\begin{proof}[Proof of Theorem ~\ref{main theorem euler}] The proof consists of three stages. We begin by combining \eqref{conjugation relationship} and \eqref{euler con exact}, along with some commutator estimates to acquire an expression of the form $\Delta \eta = D \eta M \symp \Delta f_0 + \widetilde{G}$, where $\widetilde{G}$ is of class $H^{s-1}$. We then establish that, for $v \in H_0^{s-2}$, $Mv$ being of class $H^{s-1}$ implies that $v$ is of class $H_0^{s-1}$. We finish by concluding that, as $M \symp \Delta f_0 = D \eta^{-1}\big(\Delta \eta - \widetilde{G}\big) \in H^{s-1}$, we have that $f_0 \in H^{s+2}$, and hence $u_0 \in H^{s+1}$. Then, by the results for the initial value problem, we have that $\gamma(t)$ evolves entirely in $H^{s+1}$ and the proof will be complete. Combining \eqref{conjugation relationship} and \eqref{euler con exact}:
\begin{align*}
    \Delta \eta &= D\eta \int_{0}^{1} D\gamma(t)^{-1}  R_{\gamma(t)}  P_\lambda(t) \symp \big( \Delta^{-1} R_{\gamma(t)}^{-1} \Delta f_0  \big)\ dt + D\eta \int_{0}^{1} D\gamma(t)^{-1}  R_{\gamma(t)}  P_\lambda(t) h(t) \ dt + G
\end{align*}
where again $G$ is of class $H^{s-1}$. Some straightforward calculus now yields:
\begin{align*}
    \Delta \eta &= D\eta \int_{0}^{1} D\gamma(t)^{-1}  R_{\gamma(t)}  P_\lambda(t) \Delta^{-1} R_{\gamma(t)}^{-1} \big( D\gamma(t) \symp \Delta f_0 \big) \ dt + D\eta \int_{0}^{1} D\gamma(t)^{-1}  R_{\gamma(t)}  P_\lambda(t) h(t) \ dt + G \\
    &= D\eta \int_{0}^{1} R_{\gamma(t)} \Gamma^{-1}(t) P_\lambda(t) \Delta^{-1} \Gamma(t) R_{\gamma(t)}^{-1} \symp \Delta f_0  \ dt + G 
\end{align*}
where $G$ has absorbed the term $\displaystyle D\eta \int_{0}^{1} D\gamma(t)^{-1}  R_{\gamma(t)}  P_\lambda(t) h(t) \ dt$ and is still of class $H^{s-1}$ and we denote $\Gamma:=D\gamma \circ \gamma^{-1}$ for notational simplicity. Next, we introduce commutator terms in order to achieve a more advantageous symmetry. We recall the Hodge projections:
\begin{equation*}
    \pi_{0} : T_e\D^{\sigma} \rightarrow T_e\D^{\sigma}_{\mu,ex} \oplus \nabla H^{\sigma+1} \ \text{ and } \ \pi_\mathcal{H} : T_e\D^{\sigma} \rightarrow \mathcal{H},
\end{equation*}
cf. \eqref{mean-zero and harmonic projections} and rewrite:
\begin{align*}
    \Delta \eta 
    &= D\eta \int_{0}^{1} R_{\gamma(t)}  \Delta^{-1/2} \pi_{0} P_\lambda(t)  \Delta^{-1/2}  R_{\gamma(t)}^{-1} \big( \symp \Delta f_0 \big) \ dt + \ D\eta \int_{0}^{1} R_{\gamma(t)}  \Omega(t)  R_{\gamma(t)}^{-1} \big( \symp \Delta f_0 \big) \ dt + G
\end{align*}
where
\begin{align*}
    \Omega(t) &= \Gamma^{-1}(t) P_\lambda(t) \Delta^{-1} \Gamma(t) - \Delta^{-1/2} \pi_{0} P_\lambda(t)  \Delta^{-1/2}.
\end{align*}
We claim the integral involving $\Omega(t)$ is of class $H^{s-1}$ and hence can be absorbed into $G$. Observe that $\Gamma(t)$ and its inverse have entries in $H^{s-1}$. We suppress $t$ and continue
\begin{align*}
    \Omega &= \Gamma^{-1} P_\lambda  \Delta^{-1}  \Gamma - \Delta^{-1/2} \pi_0 P_\lambda  \Delta^{-1/2} \\
    &= \Gamma^{-1} P_\lambda  \Delta^{-1}  \Gamma - P_\lambda \Delta^{-1} + P_\lambda \Delta^{-1} - \Delta^{-1/2} \pi_0 P_\lambda  \Delta^{-1/2} \\
    &= \Gamma^{-1} \big[ P_\lambda  \Delta^{-1}, \Gamma \big] + \pi_0 P_\lambda \Delta^{-1} + \pi_\mathcal{H} P_\lambda \Delta^{-1} - \Delta^{-1/2} \pi_0 P_\lambda  \Delta^{-1/2}
\end{align*}
As $\pi_0 = \Delta^{-1/2} \pi_0 \Delta^{1/2}$, we have:
\begin{align*}
    \Omega &= \Gamma^{-1} \big[ P_\lambda  \Delta^{-1}, \Gamma \big] + \Delta^{-1/2} \pi_0 \Delta^{1/2} P_\lambda \Delta^{-1} - \Delta^{-1/2} \pi_0 P_\lambda  \Delta^{-1/2} + \pi_\mathcal{H} P_\lambda \Delta^{-1} \\
    &= \Gamma^{-1} \big[ P_\lambda  \Delta^{-1}, \Gamma \big] + \Delta^{-\frac{1}{2}} \pi_0 \big[ \Delta^{\frac{1}{2}}, P_\lambda \big] \Delta^{-1} + \pi_\mathcal{H} P_\lambda \Delta^{-1}
\end{align*}
Examining the terms separately we have:
\begin{align*}
    \big[ P_\lambda  \Delta^{-1}, \Gamma \big] &= \lambda \big[ \Delta^{-1} , \Gamma \big] - [p^{ij}\partial_i\partial_j\Delta^{-1}, \Gamma] \\
    &= \lambda \big[ \Delta^{-1} , \Gamma \big] - [p^{ij} \Delta^{-1} \partial_i\partial_j, \Gamma] \\
    &= \lambda \big[ \Delta^{-1} , \Gamma \big] - p^{ij} \Delta^{-1} \partial_i\partial_j \Gamma + \Gamma p^{ij} \Delta^{-1} \partial_i\partial_j
\end{align*}
Commuting $p^{ij}$ with $\Gamma$ and introducing commutator terms, we have:
\begin{align*}
    \big[ P_\lambda  \Delta^{-1}, \Gamma \big] &= \lambda \big[ \Delta^{-1} , \Gamma \big] - p^{ij} \Delta^{-1} \partial_i\partial_j \Gamma + p^{ij} \Gamma \Delta^{-1} \partial_i\partial_j \\
    &= \lambda \big[ \Delta^{-1} , \Gamma \big] - p^{ij} \Delta^{-1} \partial_i\partial_j \Gamma + p^{ij} \Delta^{-1} \Gamma \partial_i\partial_j - p^{ij} \Delta^{-1} \Gamma \partial_i\partial_j + p^{ij} \Gamma \Delta^{-1} \partial_i\partial_j \\
    &= \lambda \big[ \Delta^{-1} , \Gamma \big] - p^{ij} \Delta^{-1} [\partial_i\partial_j, \Gamma] - p^{ij} [\Delta^{-1}, \Gamma] \partial_i\partial_j \\
    &= \lambda \Delta^{-1} \big[\Gamma, \Delta \big] \Delta^{-1} - p^{ij} \Delta^{-1} [\partial_i\partial_j, \Gamma] - p^{ij} \Delta^{-1} \big[\Gamma, \Delta \big] \Delta^{-1} \partial_i\partial_j 
\end{align*}
We can see by direct calculation of $\big[\Gamma, \Delta \big]$ and $\big[\partial_i \partial_j, \Gamma\big]$ that these terms map $H^{s-2}$ to $H^{s-3}$. As for the term $\Delta^{-\frac{1}{2}} \pi_0 \big[ \Delta^{\frac{1}{2}}, P_\lambda \big] \Delta^{-1}$ we have
\begin{align*}
    [\Delta^{\frac{1}{2}}, P_\lambda] &= [\Delta^{\frac{1}{2}}, p^{ij}]\partial_i\partial_j
\end{align*}
and, as $p^{ij}$ are of class $H^{s-1}$, the fact that this maps $H^{s}$ to $H^{s-2}$ is a consequence of the following commutator estimate, cf. \cite{taylorpseudo}:
\begin{equation}\label{kp laplacian estimate}
    \norm{\Delta^{1/2}(fg) - f\Delta^{1/2}g}_{H^{s-2}} \lesssim \norm{\nabla f}_{\infty}\norm{g}_{H^{s-2}} + \norm{f}_{H^{s-1}}\norm{g}_{\infty}
\end{equation}
Lastly, observe that the term $\pi_\mathcal{H} P_\lambda \Delta^{-1}$ maps $H^{s-2} \rightarrow C^\infty$. Hence we have
\begin{align}\label{eqn}
    \Delta \eta &= D\eta \int_{0}^{1} R_{\gamma(t)}  \Delta^{-1/2} \pi_0 P_\lambda(t)  \Delta^{-1/2}  R_{\gamma(t)}^{-1} \big( \symp \Delta f_0 \big) \ dt + \widetilde{G}
\end{align}
where now $\displaystyle \widetilde{G} = G + D\eta \int_{0}^{1} R_{\gamma(t)}  \Omega(t)  R_{\gamma(t)}^{-1} \big( \symp \Delta f_0 \big) \ dt$ evolves in $H^{s-1}$.
\par Next we consider the linear operator defined by:
\begin{equation}\label{M Euler}
Mv = \int_0^{1} M_t v \ dt := \int_{0}^{1} R_{\gamma(t)}  \Delta^{-1/2} \pi_0 P_\lambda(t)  \Delta^{-1/2}  R_{\gamma(t)}^{-1} v \ dt
\end{equation}
The boundedness of $M$ on $H^{k}_0$ for $0\leq k \leq s-1$ follows from the uniform boundedness in $t\in[0,1]$ of the operators comprising the integrand $M_t$. Notice now that \eqref{eqn} can be rearranged as:
\begin{equation}\label{M eqn}
    M(\symp \Delta f_0) = D\eta^{-1} \big( \Delta \eta - \widetilde{G} \big) \in H^{s-1}(\T^2, \R^2)
\end{equation}
Our goal is now to show that, for $0 \leq k \leq s-1$, if $M v \in H^{k}$, then $v \in H^k_0$.
For $k=0$ we have from Lemma \ref{elliptic estimate} and the $L^2$-orthogonality in the Hodge decomposition \eqref{hodge} that:
\begin{align*}
    \norm{M v}_{L^2}\norm{v}_{L^2} &\geq \big\vert \ip{Mv, v}_{L^2} \big\vert \\
    &= \bigg\vert \int_0^1 \ip{\pi_0 P_\lambda(t)  \Delta^{-1/2}  R_{\gamma(t)}^{-1}v \ , \  \Delta^{-1/2}  R_{\gamma(t)}^{-1}v}_{L^2} \ dt \bigg\vert \\
    &\gtrsim \bigg\vert \int_0^1 \norm{  \Delta^{-1/2}  R_{\gamma(t)}^{-1}v}^2_{H^1} \ dt \bigg\vert \\
    &\simeq \norm{v}^2_{L^2}
\end{align*}
From which it follows that if $Mv \in L^2 $, then $v \in L^2$. Now assume the lemma holds for some $k$ with $0 \leq k \leq s-2$ and $Mv \in H^{k+1}$. By the inductive hypothesis, $v \in H^{k}_0$. Furthermore, for $j = 1, 2$ we have:
\begin{equation*}
    M\partial_j v = \partial_j M v + [M, \partial_j] v
\end{equation*}
By assumption, $\partial_j M v \in H^{k}$. As for the latter term, we show that for $0 \leq k \leq s-2$, we have $[M_t, \partial_j] H^k_0 \rightarrow H^k_0$, and hence $[M, \partial_j]: H^k_0 \rightarrow H^k_0$. Expanding yields:
\begin{align*}
    [M_t, \partial_j] &= C_\gamma \big( \big[ \Delta^{-1/2} \pi_0 P_\lambda \Delta^{-1/2}, \widetilde{C}_\gamma(\partial_j) \big] \big)
\end{align*}
where again $C_{\gamma}( B ) := R_{\gamma}  B  R_{\gamma}^{-1}$ and $\widetilde{C}_{\gamma}( B ) := R_{\gamma}^{-1}  B  R_{\gamma}$. Examining the term inside the conjugation above, we have:
\begin{align*}
    \big[ \Delta^{-1/2} \pi_0 P_\lambda \Delta^{-1/2}, \widetilde{C}_\gamma(\partial_j)\big] &= \Delta^{-1/2} \pi_0 P_\lambda \Delta^{-1/2} \widetilde{C}_\gamma(\partial_j) - \widetilde{C}_\gamma(\partial_j) \Delta^{-1/2} \pi_0 P_\lambda \Delta^{-1/2} \\
    &= \Delta^{-1/2} \pi_0 P_\lambda \Delta^{-1/2} \widetilde{C}_\gamma(\partial_j) - \Delta^{-1/2} \pi_0 P_\lambda \widetilde{C}_\gamma(\partial_j) \Delta^{-1/2} \\
    & \qquad + \Delta^{-1/2} \pi_0 P_\lambda \widetilde{C}_\gamma(\partial_j) \Delta^{-1/2} - \widetilde{C}_\gamma(\partial_j) \Delta^{-1/2} \pi_0 P_\lambda \Delta^{-1/2} \\
    &= \Delta^{-1/2} \pi_0 P_\lambda \big[ \Delta^{-1/2}, \widetilde{C}_\gamma(\partial_j)\big] + \big[ \Delta^{-1/2} \pi_0 P_\lambda, \widetilde{C}_\gamma(\partial_j) \big] \Delta^{-1/2} \\
    &= \Delta^{-1/2} \pi_0 P_\lambda  \Delta^{-1/2} \big[\widetilde{C}_\gamma(\partial_j), \Delta^{1/2}\big] \Delta^{-1/2} + \big[ \Delta^{-1/2} \pi_0 P_\lambda, \widetilde{C}_\gamma(\partial_j) \big] \Delta^{-1/2}
\end{align*}
Notice that the term
\begin{equation*}
    \big[\widetilde{C}_\gamma(\partial_j), \Delta^{1/2}\big] = \big[\partial_j \gamma^i \partial_i, \Delta^{1/2}\big] = \big[\partial_j \gamma^i, \Delta^{1/2}\big]\partial_i
\end{equation*}
maps $H^{k+1}_0 \rightarrow H^{k}_0$ for any $0\leq k \leq s-2$. {Furthermore we have:}
\begin{align*}
    \big[ \Delta^{-1/2} \pi_0 P_\lambda, \widetilde{C}_\gamma(\partial_j) \big] &= \Delta^{-1/2} \pi_0 P_\lambda \widetilde{C}_\gamma(\partial_j) - \widetilde{C}_\gamma(\partial_j) \Delta^{-1/2} \pi_0 P_\lambda \\
    &= \Delta^{-1/2} \pi_0 P_\lambda \widetilde{C}_\gamma(\partial_j) - \Delta^{-1/2} \pi_0 \widetilde{C}_\gamma(\partial_j) P_\lambda + \Delta^{-1/2} \pi_0 \widetilde{C}_\gamma(\partial_j) P_\lambda - \widetilde{C}_\gamma(\partial_j) \Delta^{-1/2} \pi_0 P_\lambda \\
    &= \Delta^{-1/2} \pi_0 \big[P_\lambda, \widetilde{C}_\gamma(\partial_j)\big] + \big[\Delta^{-1/2}\pi_0, \widetilde{C}_\gamma(\partial_j)\big]P_\lambda
\end{align*}
The commutator term $\big[P_\lambda, \widetilde{C}_\gamma(\partial_j)\big]$ can be explicitly calculated to be a second order operator with $H^{s-3}$ coefficients. {As for $\big[\Delta^{-1/2}\pi_0, \widetilde{C}_\gamma(\partial_j)\big]P_\lambda$, we compute:}
\begin{align*}
    [\Delta^{-1/2}\pi_0, \widetilde{C}_\gamma(\partial_j)\big] &= \Delta^{-1/2} \pi_0 \widetilde{C}_\gamma(\partial_j) - \widetilde{C}_\gamma(\partial_j) \Delta^{-1/2} \pi_0 \\
    &= \Delta^{-1/2} \pi_0 \widetilde{C}_\gamma(\partial_j)(\pi_0 + \pi_\mathcal{H}) - \widetilde{C}_\gamma(\partial_j) \Delta^{-1/2} \pi_0
\end{align*}
To proceed we introduce the identity on $H^{k}_0$ as $\Delta^{1/2} \Delta^{-1/2}$ into the term involving the projection $\pi_0$. So we have:
\begin{align*}
    [\Delta^{-1/2}\pi_0, \widetilde{C}_\gamma(\partial_j)\big] &= \Delta^{-1/2} \pi_0 \widetilde{C}_\gamma(\partial_j)\Delta^{1/2} \Delta^{-1/2}\pi_0 - \widetilde{C}_\gamma(\partial_j) \Delta^{-1/2} \pi_0 + \Delta^{-1/2} \pi_0 \widetilde{C}_\gamma(\partial_j) \pi_\mathcal{H} \\
    &= \Delta^{-1/2} \big[ \widetilde{C}_\gamma(\partial_j), \Delta^{1/2}\big]\Delta^{-1/2}\pi_0 + \Delta^{-1/2} \pi_0 \widetilde{C}_\gamma(\partial_j) \pi_\mathcal{H}
\end{align*}
As before $\big[ \widetilde{C}_\gamma(\partial_j), \Delta^{1/2}\big]: H^{k+1}_0 \rightarrow H^{k}_0$, and we have $[M, \partial_j]: H^k_0 \rightarrow H^k_0$ for $0 \leq k \leq s-2$. Hence we have shown that $M\partial_j v = \partial_j M v + [M, \partial_j] v \in H^{k}$. So by the inductive hypothesis we have $\partial_j v \in H^k_0$ which gives us $v \in H^{k+1}_0$. Therefore, we have that, for $0 \leq k \leq s-1$, if $M v \in H^{k}$, then $v \in H^k_0$.
\par Finally, since \eqref{M eqn} says
\begin{equation*}
    M(\symp \Delta f_0) \in H^{s-1}
\end{equation*}
we have
\begin{equation*}
    \symp \Delta f_0 \in H^{s-1}
\end{equation*}
which implies that
\begin{equation*}
    \symp f_0 \in H^{s+1}
\end{equation*}
which finally gives us that $u_0 \in H^{s+1}$ and the proof is complete.
\end{proof}
\begin{corollary}
For $s>6$, if $\gamma(t)$ is an $L^2$ geodesic on $\D^s_\mu(\T^2)$ with, for some $T>0$, $\gamma(0) = \gamma(T) = e$, then $\gamma(t)$ evolves entirely in $\D_\mu(\T^2)$. In other words, any isochronal Euler flow on $\T^2$ must either be smooth or of some `low-regularity' in a $H^s$ sense.
\end{corollary}
\subsubsection{Higher Order Euler-Arnold Equations}
In this section we equip the group of volumorphisms with a right-invariant $H^r$ metric \eqref{inertia} induced by the interia operator $A^r$, for integer $r\geq 1$, then \eqref{Lie group geodesic equation} becomes:

\begin{equation}\label{higher order EA}
    \begin{cases}
    \partial_t A^ru + \nabla_u A^r u + \big( \nabla u \big)^{\top}A^r u = -\nabla p \\
    \text{div}(u) = 0\\
    \end{cases}
\end{equation}
The main result in this section is as follows:

\begin{theorem}\label{main theorem higher order EA}
Let $r \geq 1$ be an integer, $s > 2r + 6$ and consider the group of volume-preserving diffeomorphisms of $\T^2$ equipped with a right-invariant $H^r$ metric \eqref{inertia}. Given $u_0 \in T_e\D^s_\mu$ let $\gamma(t)$ denote the corresponding $H^r$ geodesic. If at time $t=1$, $\gamma$ passes through a point $\eta \in \D^{s+1}_\mu$, then we have $u_0 \in T_e\D^{s+1}_\mu$ and consequently $\gamma(t)$ evolves entirely in $\D^{s+1}_\mu$.
\end{theorem}

Taking the curl of both sides of \eqref{higher order EA}, we acquire a ``conservation of vorticity"-type equation in this context.
\begin{equation}\label{H^r conservation of vorticity}
    \rot(A^r u) \circ \gamma = \rot (A^r u_0)
\end{equation}

This allows us to derive an analogous conservation law to \eqref{euler con exact}.
\begin{lemma}
 For $u_0 = \symp f_0 + h$, $\gamma$ and $u$ as above, we have:
\begin{equation}\label{H^r con exact}
    u(t) = \symp \big( \Delta^{-1} A^{-r} R_{\gamma(t)}^{-1} A^r \Delta f_0  \big) + h(t)
\end{equation}
where $h(t)$ evolves in $\mathcal{H}$.
\end{lemma}
\begin{proof}
Again, as in the $L^2$ setting, applying the symplectic gradient $\symp:=(-\partial_2, \partial_1)$ to both sides of \eqref{H^r conservation of vorticity} gives us:
\begin{equation*}
    A^r \Delta u = \Ad_\gamma A^r \Delta u_0
\end{equation*}
Which immediately yields \eqref{H^r con exact}.
\end{proof}
\begin{proof}[Proof of Theorem ~\ref{main theorem higher order EA}]
Using Lemma \ref{workhorse} and \eqref{H^r con exact}, this follows from a completely analogous argument as in Theorem \ref{main theorem euler}.
\end{proof}
\begin{remark}
If we define $A = 1 - \alpha^2 \Delta$, the above theorem covers the case of the Euler-$\alpha$ equations studied in \cite{hmr} and \cite{shkoller}.
\end{remark}
\subsection{Swirl-Free Axisymmetric Diffeomorphisms of the Flat 3-Torus $\T^3$}\label{axisymmetric}
In this section we consider the group of axisymmetric diffeomorphisms of the flat periodic box $\T^3$, equipped with any of the Killing fields $K = \frac{\partial}{\partial{x_i}}$ and a certain subclass of \textit{swirl-free} initial data.
\par Given an axisymmetric vector field $v$, we define its swirl to be the function $g(v, K)$ on $\T^3$. A vector field is called swirl-free if this function identically vanishes. In our argument for this section, the swirl-free condition will deliver a conservation law which will play the same role that conservation of vorticity did in the previous section. It is shown in \cite{lmp} that the swirl of an axisymmetric velocity field is transported by its flow. More precisely, if $u_0 \in T_e\A_\mu^s$ and $\gamma(t)$ is the corresponding geodesic in $\A_\mu^s$ then $g(u,K) \circ \gamma(t) = g(u_0,K)$ as long as it is defined. We denote the space of swirl-free axisymmetric vector fields by $T_e\A_{\mu,0}^s$. The proof of the following lemma can be found in \cite{lmp}.
\begin{lemma}\label{curl lemma} Let $K$ be any of the Killing fields $\partial_i$. Then:
\begin{enumerate}
\item If $v \in T_e\A_{\mu,0}^{s+1}$, then $\curl v = \phi K$, where $\phi$ is a function of class $H^s$.
\item If $u_0 \in T_e \A_{\mu,0}^s$ and $u(t)$ is the corresponding solution  of the Euler equations \eqref{Euler}, then, by the above, we can write $\curl u_0 = \phi_0 K$ and $\curl u(t,x) = \phi(t,x)K(x)$. The function $\phi$ is transported along the flow lines: $\phi(t,\gamma(t)) = \phi_0(x)$.
\item If $u_0 \in T_e \A_{\mu,0}^s$ then the corresponding solution $u(t)$ of the Euler equations \eqref{Euler} can be extended globally in time.
\item If $u_0$ and $u(t)$ are as above, and $\gamma(t) \in \A^{s}_{\mu,0}$ is the flow of $u(t)$, then $K$ is preserved by the adjoint: $\Ad_{\gamma(t)}K = K$.
\item For $v \in T_p\T^3$ with $g\big(v, K(p)\big) = 0$, we have $g\big(D\gamma(t) v , K(\gamma(t,p))\big) = 0$
\end{enumerate}
\end{lemma}
We now proceed to the main result for this section:
\begin{T2}
Let $s>\frac{13}{2}$ and consider the group of axisymmetric diffeomorphisms of $\T^3$ with respect to any of the Killing fields $K=\partial_i$, equipped with a right invariant $L^2$ metric. Given $u_0 \in T_e\A^s_{\mu,0}$ let $\gamma(t)$ denote the corresponding $L^2$ geodesic. If at time $t=1$, $\gamma$ passes through a point $\eta \in \A^{s+1}_\mu$, then we have $u_0 \in T_e\A^{s+1}_{\mu, 0}$ and consequently $\gamma(t)$ evolves entirely in $\A^{s+1}_\mu$.
\end{T2}
\par Without loss of generality, we will assume, for the duration of the section, that $K=\partial_3$. So, if we assume a vector field is axisymmetric and swirl-free, this is now equivalent to saying $v = v_1\partial_1 + v_2 \partial_2$, where $\partial_3v_1 = \partial_3v_2 = 0$. Proceeding as before we establish a conservation law:
\begin{lemma}\label{axisymmetric laplacian conservation}
Let $u(t)$ be the Eulerian velocity of the flow $\gamma(t)$. Then we have:
\begin{equation}\label{flat axisymmetric con}
    \Delta u (t) = \Ad_{\gamma(t)}  \Delta u_0.
\end{equation}
\end{lemma}
\begin{proof}
As $u(t)$ evolves in $T_e\A^s_{\mu, 0}$, we may write $u(t) =u_1(t)\partial_1 + u_2(t)\partial_2$, where $\partial_3u_1(t) = \partial_3u_2(t) = 0$. Hence by taking the curl we have:
\begin{align*}
    \curl u (t) &= \big(-\partial_2u_1(t) + \partial_1u_2(t)\big) \partial_3.
\end{align*}
\par Comparing this with Lemma \ref{curl lemma}, we see $\phi = -\partial_2u_1(t) + \partial_1u_2(t)$ and thus, as $\phi$ is preserved along the flow lines ($\phi \circ \gamma = \phi_0$), we have as in 2D:
\begin{equation}
    \big(-\partial_2u_1(t) + \partial_1u_2(t)\big) \circ \gamma = -\partial_2u_{0,1}(t) + \partial_1u_{0,2}(t).
\end{equation}
Since
$$\Delta u_0 = \Delta u_{0,1} \partial_1 + \Delta u_{0,2} \partial_2,$$
$$\Delta u(t) = \Delta u_{1}(t) \partial_1 + \Delta u_{2}(t) \partial_2,$$
and $$D \gamma(t) = \begin{bmatrix}
\partial_1\gamma_1 & \partial_2\gamma_1 & 0\\
\partial_1\gamma_2 & \partial_2\gamma_2 & 0\\
0 & 0 & 1
\end{bmatrix},$$
the lemma follows in an identical fashion to the computation for \eqref{euler con}.
\end{proof}
We proceed with the proof of Theorem \ref{main theorem axisymmetric}: 
\begin{proof}[Proof of Theorem ~\ref{main theorem axisymmetric}]
Using \eqref{relationship} and \eqref{flat axisymmetric con}, this follows from a completely analogous argument as in Theorem \ref{main theorem euler}.
\end{proof}

\begin{remark}
{Further work in this direction will include the study of curved manifolds, possibly with boundary. As mentioned previously, many of our constructions extend to the setting of curved spaces by collecting any lower order terms arising in the various calculations due to derivatives of the components of the metric and its Christoffel symbols on $M$ into a single term which is negligible for our purposes. As an initial case, the author has established the result for swirl-free initial data for certain Killing fields on the round 3-sphere $\mathbb{S}^3$.}
\newline \indent {The strategy for manifolds with boundary will have to be modified however. For example, the estimates involved need to be reproved and the presence of boundary terms when integrating by parts would have to be accounted for. The calculations involving non-local operators $\Delta^{-1}$, etc. will have to be readdressed. Preliminary calculations in the case of axisymmetric flows on the vertically periodic cylinder $D^2 \times \mathbb{S}^1$ with Killing field $\partial_\theta$ suggest that this can be done; for instance the ``inverted" version of \eqref{Lie group conservation law equation} obtained in this setting matches that of the version involved in Theorem \ref{main theorem axisymmetric} up to negligible lower order terms.}
\newline \indent {Also of interest is the case axisymmetric flows with sufficiently small swirl, as considered in \cite{lmp}.}
\end{remark}

\subsection{Symplectomorphisms of the Flat Torus $\T^{2k}$}\label{symplectomorphisms}
\par In this section we consider the group of symplectomorphisms of the torus $\T^{2k}$, equipped with the standard symplectic form $\omega$. The main result presented in this section is 

\begin{theorem}\label{main theorem symplectomorphisms}
Let $s>k+5$ and consider the group of symplectomorphisms of $\T^{2k}$ equipped with a right-invariant $L^2$ metric. Given $u_0 \in T_e\D^s_\omega$ let $\gamma(t)$ denote the corresponding $L^2$ geodesic. If at time $t=1$, $\gamma$ passes through a point $\eta \in \D^{s+1}_\omega$, then we have $u_0 \in T_e\D^{s+1}_\omega$ and consequently $\gamma(t)$ evolves entirely in $\D^{s+1}_\omega$.
\end{theorem}

\par As before we begin with a consequence of \eqref{Lie group conservation law equation} and establish a new conservation law from it. The following can be found in \cite{ebin2012geodesics}.

    \begin{equation}\label{coadjoint con symplectomorphisms}
        \gamma^*\delta\omega^\flat u = \delta \omega^\flat u_0
    \end{equation}
Using this we obtain:
\begin{proposition}
    \begin{equation}\label{con law symplectomorphisms}
        \Delta u = \Ad_\eta \Delta u_0
    \end{equation}
\end{proposition}
\begin{proof}
Applying $g^\sharp d$ to both sides of \eqref{coadjoint con symplectomorphisms}:
        \begin{align*}
            g^\sharp \eta^* \big( \Delta \omega^\flat u \big) &= g^\sharp \Delta \omega^\flat u_0\\ 
            \Rightarrow D\eta^\top  g^\sharp \omega^\flat \big( \Delta u \circ \eta \big) &= g^\sharp \omega^\flat \Delta u_0 \\ 
            \Rightarrow D\eta^\top  J \big( \Delta u \circ \eta \big) &= J \Delta u_0 \\ 
            \Rightarrow \Delta u &= R_\eta^{-1} J^{-1}  \big(D\eta^\top\big)^{-1}  J \Delta u_0 \\
            &=  R_\eta^{-1} D\eta \Delta u_0 \\ 
            &= \Ad_\eta \Delta u_0
        \end{align*}
        Where, in the penultimate line, we have used the fact that $D_\eta$ is a symplectic matrix.
\end{proof}

\begin{proof}[Proof of Theorem ~\ref{main theorem symplectomorphisms}]
This follows in a completely analogous fashion to the proof of Theorem \ref{main theorem euler}, using the above conservation law \eqref{con law symplectomorphisms}.
\end{proof}

\section{The Frech{\'e}t Setting}\label{smooth}
As mentioned in the introduction, to the best of the author's knowledge, the earliest investigations into the kind of regularity property considered in this paper are due to Constantin \& Kolev \cite{ck2002}, \cite{ck2003} and Kappeler, Loubet \& Topalov \cite{klt1d}, \cite{klt}. In each instance, the authors used this property to construct exponential maps which were local `$C^1_F$-diffeomorphisms' in the Frech{\'e}t category. We obtain analogous results here for the settings we have considered. As in the above, we make use of a Nash-Moser-type inverse function theorem for Frech{\'e}t spaces admitting Hilbert approximations cf. \cite{hamilton} \& \cite[Theorem~A.5]{klt}. First we recall some basic defintions.
\par Let $X$ and $Y$ be Frech{\'e}t spaces with $U \subseteq X$ and $V \subseteq Y$ open subsets. We say $f: U \rightarrow Y$ is differentiable at $u \in U$ in the direction $x \in X$ if the limit 
\begin{equation*}
    \displaystyle \delta_{u}f(x) := \lim_{\varepsilon \to 0} \frac{f(u + \varepsilon x) - f(u)}{\varepsilon}
\end{equation*}
converges in $Y$ with respect to the Frech{\'e}t topology. If $\delta_uf(x)$ exists for all $(u,x) \in U \times X$ and the map given by
\begin{equation*}
\begin{split}
    \delta f : \ &U \times X \rightarrow Y \\ : \ &(u,x) \mapsto \delta_uf(x)
\end{split}
\end{equation*}
is continuous with respect to the Frech{\'e}t topologies on $U \times X$ and $Y$, then $f$ is said to be continuously differentiable\footnote{It is important to note that this notion of continuously differentiable is weaker than the standard definition, cf. \cite[Page 73]{hamilton}} or $C^{1}_F$. A map $f:U \rightarrow V$ is called a $C^1_F$-diffeomorphism if it is a homeomorphism where both $f$ and $f^{-1}$ are $C^1_F$. We now state our version of the inverse function theorem:
\begin{lemma}\label{inverse function theorem}
Let $s_0>\frac{n}{2}+1$ and assume we have a well-defined exponential map on $\big(\G^{s_0}, \ip{ \cdot , \cdot }\big)$:
\[\exp^{s_0} : U^0 \rightarrow V^0\]
which is a $C^1$-diffeomorphism, where $U^0$ is a neighbourhood of $0 \in T_e\G^{s_0}$ and $V^{0}$ is a neighbourhood of $e \in \G^{s_0}$, lying in the image of one of our chart maps for $\G^{s_0}$. Define, for any $k \in \Z_{\geq0}$, $s_k:=s_0+k$, $U^k := U^0 \cap T_e\G^{s_0 + k}$, $U := U^{0} \cap T_e\G$, $V^k := V^0 \cap \G^{s_0 + k}$, $V := V^{0} \cap \G$ and $\exp^{s_k} := \exp^{s_0}\big\vert_{U^k}$. If, for all $k$, the following properties hold:
\begin{enumerate}
    \item $\exp^{s_k}: U^k \rightarrow V^k$ is a bijective $C^1$-map;
    \item For any $u \in U$, $d_u \exp^{s_0} : T_e\G^{s_0} \rightarrow T_{\exp^{s_0}(W)}\G^{s_0}$ is a linear isomorphism with the property that 
    \[d_u \exp^{s_0} (T_e\G^{s_k} \setminus T_e\G^{s_{k+1}}) \subseteq T_{\exp^{s_0}(u)}\G^{s_k}\setminus T_{\exp^{s_0}(u)}\G^{s_{k+1}}.\]
\end{enumerate}
Then $\exp := \exp^{s_0}\big\vert_{U} : U \rightarrow V$ is a $C^1_F$-diffeomorphism.
\end{lemma}

\begin{proof}
By $(1)$, $\exp: U \rightarrow V$ is a well-defined bijection. Examining the derivatives, we note that, for any $u \in U$, $d_u \exp^{s_k} = d_u\exp^{s_0} \big\vert_{T_e\G^{s_k}}$. Hence, by our assumptions, the map $d_u \exp^{s_k} : T_e \G^{s_k} \rightarrow T_{\exp^{s_0}(u)}\G^{s_k}$ is a bounded linear bijection in the $H^{s_k}$ topologies and hence, by the open mapping theorem, is a linear isomorphism. Now, as $U$ is a dense subset of each $U^k$ in the $H^{s_k}$ topology, by applying the inverse function theorem at each point $u \in U$, we have ${\exp^{s_k}}^{-1} : V^k \rightarrow U^k$ is a $C^1$ map in the $H^{s_k}$ topologies.
\par So, for any $k \in \Z_{\geq0}$, $u\in U$ and $\xi \in V$, we have:
\[\exp = \exp^{s_0}\big\vert_{U} = \exp^{s_k}\big\vert_{U} \ ,\]
\[\delta_u\exp = d_u \exp^{s_k} \big\vert_{T_e\G}\] and
\[ \delta_\xi \big(\exp^{-1}\big) = d_\xi \big({\exp^{s_k}}^{-1}\big)\big\vert_{T_{\xi}\G}\]
Hence, $\delta \exp : U \times T_e\G \rightarrow TV$ and $\delta \big( \exp^{-1} \big) : TV \rightarrow U \times T_e\G$ are continuous in the Frech{\'e}t topologies.
\par Hence, $\exp: U \rightarrow V$ is a $C^{1}_F$ diffeomorphism.
\end{proof}
{We now state the main theorem for this section, which encompasses Theorem \ref{frechet T2} and Theorem \ref{frechet T3} from the introduction.}
\begin{theorem}\label{euler frechet}
For each of the settings we have considered in Section \ref{main}, we can construct an exponential map on a neighbourhood of $0 \in T_e\G$ which is a $C^1_F$ diffeomorphism onto its image.
\end{theorem}

\begin{remark}
{We will prove Theorem \ref{euler frechet} for the case of $\D_\mu(\T^2)$ equipped with the $L^2$ metric. We note that this case is not a new result, cf. \cite{omori1973}, however we have used a different method of proof. The analogous results in the other settings considered in Section \ref{main} follow from an analogous argument, making use of the literature pertaining to Fredholmness of exponential maps on groups of diffeomorphisms cf. \cite{benn}, \cite{bmp}, \cite{bkp2016}, \cite{lmp} and \cite{mp}. We will use the notation defined in Lemma \ref{inverse function theorem}.}
\end{remark}
\begin{proof} Let $s_0 > 6$. From \cite{em}, we have that $\D^{s_0}_\mu(\T^2)$ equipped with the $L^2$ metric admits a well-defined exponential map which is a local $C^\infty$ diffeomorphism at the identity $\exp^{s_0}: U^{s_0} \rightarrow V^{s_0}$. We may shrink $U^{s_0}$ if necessary so that $V^{s_0}$ lies in the image of a chart map for $\D^{s_0}_\mu(\T^2)$.
\par We know from \cite[Theorem~12.1]{em} that, for all $k \in \Z_{\geq0}$, $\exp^{s_0}(U^k) \subseteq V^k$. Furthermore, uniqueness and smooth dependence of Lagrangian solutions on initial data in each $H^{s_k}$ topology gives us that $\exp^{s_k} := \exp^{s_0}\big\vert_{U^k} : U^k \rightarrow V^k$ is a well-defined $C^\infty$ injection. Theorem \ref{main theorem euler} now guarantees that, for all $k\in \Z_{\geq0}$, $\exp^{s^k} : U^k \rightarrow V^k$ is in fact a $C^{\infty}$ bijection.
\par Next, from \cite{emp}, we know that $\exp^{s_0}$ is a non-linear Fredholm map of index zero. Hence, we can further restrict $U^{s_0}$ if necessary to guarantee that, for all $u \in U$, $d_u \exp^{s_0} : T_e \D_\mu^{s_0}(\T^2) \rightarrow T_{\exp^{s_0}(u)} \D_\mu^{s_0}(\T^2)$ is a linear isomorphism. Furthermore, defining $\eta:=\exp^{s_0}(u) \in \D_\mu(\T^2)$ we in fact have:
\begin{equation*}
    d_u \exp^{s_0} = D \eta \big( \Omega_u - \Gamma_u \big)
\end{equation*}
where, for any $k \in \Z_{\geq0}$, $D\eta : T_e\D_\mu^{s_k}(\T^2) \rightarrow T_\eta \D_\mu^{s^k}(\T^2)$ and $\Omega_u : T_e\D_\mu^{s_k}(\T^2) \rightarrow T_e\D_\mu^{s_k}(\T^2)$ are linear isomorphisms and $\Gamma_u : T_e\D_\mu^{s_k}(\T^2) \rightarrow T_e\D_\mu^{s_{k+1}}(\T^2) \subset T_e\D_\mu^{s_k}(\T^2)$ is a compact operator. So, if for $w \in T_e \D_\mu^{s_k}(\T^2)$, $d_u\exp^{s_0}(w) \in T_\eta\D_\mu^{s_{k+1}}(T^2)$, we have that $w = \Omega_u^{-1}\bigg( D\eta^{-1}\big(d_u\exp^{s_0}(w)\big) + \Gamma_u(w) \bigg) \in T_e\D_\mu^{s_{k+1}}(\T^2)$.
\par Hence, $d_u\exp^{s_0}(T_e\D_\mu^{s_k}(\T^2) \setminus T_e\D_\mu^{s_{k+1}}(\T^2)) \subseteq T_\eta\D_\mu^{s_k}(\T^2) \setminus T_\eta\D_\mu^{s_{k+1}}(\T^2)$ and we may apply Lemma \ref{inverse function theorem}.
\end{proof}

\begin{remark}
It is important to note that Theorem \ref{euler frechet} does not follow immediately from the work of Ebin and Marsden in \cite{em}. While they define an exponential map for each Sobolev index $s>\frac{n}{2}+1$, $\exp^s : \widetilde{U}^s \rightarrow \widetilde{V}^s$ and, indeed, their Theorem 12.1 ensures that each $\exp^s$ will map smooth initial data to a geodesic in $\D_\mu(\T^2)$, they do so by applying the inverse function theorem in separately for each index. Hence, there is no apriori relationship between $\widetilde{U}^s$ which guarantees that their intersection is not a single point, cf. \cite[page 87]{omori1973}.
\end{remark}

\section*{Funding} The author was supported in part by the National University of Ireland Dr. {\'E}amon de Valera Travelling Studentship in Mathematics.

\section*{Acknowledgments} The bulk of this work was completed during the author's Ph.D. studies at the University of Notre Dame. The author wishes to thank their advisor, Professor Gerard Misio{\l}ek, for introducing them to the problem and for many inspiring conversations.

\bibliographystyle{amsplain}

\end{document}